\documentclass[3p,number]{elsarticle}
\usepackage{color,mathdots,amsthm,amsmath,times}
\usepackage[numbers]{natbib}

\usepackage[colorlinks, linkcolor=reference,
            citecolor=citation,
           urlcolor=e-mail]{hyperref}
\definecolor{e-mail}{rgb}{0,.40,.80}
\definecolor{reference}{rgb}{.20,.60,.22}
\definecolor{citation}{rgb}{0,.40,.80}
\newtheorem{theorem}{Theorem}[section]
\newtheorem{proposition}[theorem]{Proposition}
\newtheorem{lemma}[theorem]{Lemma}
\newtheorem{corollary}[theorem]{Corollary}

\theoremstyle{definition}
\newtheorem{definition}[theorem]{Definition}

\theoremstyle{remark}
\newtheorem{remark}[theorem]{Remark}
\newtheorem{example}[theorem]{Example}

\numberwithin{equation}{section}

\usepackage{amsmath,amssymb}

\newcommand{\Ge}{\geqslant}
\newcommand{\ZN}{{\mathbb Z}_{\Ge 0}}

\def \L {\Lambda}

\def \NN {{\mathbb Z}_{\Ge 0}}

\def \L {\mathfrak L}

\definecolor{todo}{rgb}{0,0,1}

\begin{document}
\begin{frontmatter}
\title{On bounds for the effective differential Nullstellensatz}
\author{Omar Le\'on S\'anchez} 
\ead{oleonsan@math.mcmaster.ca} 
\address{Department of Mathematics and Statistics, McMaster University,
1280 Main St W,
Hamilton, ON L8S 4L8, Canada}
\author{Alexey Ovchinnikov\fnref{AO}}
\ead{aovchinnikov@qc.cuny.edu}
\address{Department of Mathematics,
CUNY Queens College,
65-30 Kissena Blvd,
Queens, NY 11367, USA\\
Ph.D. Program in Mathematics, CUNY Graduate Center,
365 Fifth Avenue, New York, New York  10016, USA}

\begin{abstract}
Understanding bounds for the effective differential Nullstellensatz is a central problem in differential algebraic geometry. Recently, several bounds have been obtained using Dicksonian and antichains sequences (with a given growth rate). In the present paper, we make these bounds more explicit and, therefore, more applicable to understanding the computational complexity of the problem, which is essential to designing more efficient algorithms. 
\end{abstract}

\fntext[AO]{A. Ovchinnikov was partially supported by the NSF grants CCF-0952591 and DMS-1413859.}

\end{frontmatter}

\section{Introduction}

The effective differential Nullstellensatz problem can be stated as follows: Given a system of algebraic partial differential equations $F=0$ where $F=f_1,\dots,f_s$, can one effectively determine if the system is consistent? In other words, is there an effective procedure to determine if 1 belongs or not to the differential ideal generated by $F$ in the ring of differential polynomials? To determine if 1 belongs to an ideal in a polynomial ring, one can use algebraic effective methods (for instance, \cite{boundsNull1, boundsNull2}). Thus, the problem reduces to finding an effective bound $B$ such that 1 is in the differential ideal generated by $F$ if and only if 1 is in the ideal generated by $F$ and its derivatives of order at most $B$. 

Let us rephrase the above problem in more technical terms. Let $m, n, \ell,  D$ be positive integers. An upper bound for the effective differential Nullstellensatz is an effectively determined function $B=B(m,n,\ell,D)$ that is minimal with respect to the following property: For any differential field $(K,\partial_1,\ldots,\partial_m)$ of characteristic zero with $m$ commuting derivations, and any finite set $F\subset K\{x_1,\dots,x_{n}\}$ of differential polynomials over $K$ in $n$ differential indeterminates of order and degree bounded by $\ell$ and $D$, we have
$$1\in [F] \iff 1\in (F)^{(B)}.$$
Here, $[F]$ denotes the differential ideal generated by $F$, and $(F)^{(B)}$ the ideal generated by $F$ together with its derivatives up to order $B$. 

In order to determine the bound $B$ using a differential elimination algorithm, %\cite{Bou,Hu} 
one needs to determine how many differentiation steps the algorithm makes. Determining this number of steps is the main difficulty of the problem. The first attempt to a solution was given by Seidenberg \cite{Seid} in 1956, where it was suggested how this bound could be obtained. In \cite[Theorem 1]{DiffNull}, using bounds on the length of \emph{Dicksonian} sequences, an explicit bound was found in terms of the Ackermann function (see Section \ref{sec:dicksonian} for the recursive definition of this function). More precisely, they proved that 
\begin{equation}\label{oldbound}
B\leq A(m+8,n+\max(n,\ell,D)).
\end{equation}
Recently, in \cite[Theorem 3.4]{DiffNull2}, a better bound was found. More explicitly,
\begin{equation}\label{newbound}
B\leq (n\alpha_{T-1}D)^{2^{O\big(n^3\alpha_T^3\big)}},
\end{equation}
where $\alpha_T=\binom{T+m}{m}$ and $T=T(m,n,\ell)$ is defined below. 

In order to say what the value of $T=T(m,n,\ell)$ is, we need the following terminology. Consider the order $\leq$ on $\ZN^m\times n$ defined as $(\tau,i)\leq (\eta,j)$ iff $i=j$ and $\tau$ is less than or equal to $\eta$ in the product order of $\ZN^m$. To be clear $n=\{0,1,\dots,n-1\}$, and so $(\tau,i)\in\ZN^m\times n$ means that $\tau\in\ZN^m$ and $0\leq i\leq n-1$. If $a\in \NN^m\times n$, we let the degree of $a$ be $\deg a=\deg\tau:=\tau_1+\dots +\tau_m$ where $a=(\tau,i)$ and $\tau=(\tau_1,\dots,\tau_m)$. A sequence $a_1,a_2,\dots$ of $\ZN^m\times n$ is called \emph{Dicksonian} if for all $i<j$ we have $a_i\nleq a_j$. By Dickson's lemma, every Dicksonian sequence is finite. On the other hand, an \emph{antichain} sequence is a Dicksonian sequence with the additional property that $a_i\ngeq a_j$ for $i<j$; in particular, an antichain sequence, as a set, is an antichain of $\ZN^m\times n$ with respect to the order $\leq$. In the case $n=1$, we simply write $\ZN^m$ instead of $\ZN^m \times 1$, so it is clear what is meant by a Dicksonian (or antichain) sequence of $\ZN^m$.

Given a function $f\colon \ZN\to \ZN$, we say that the degree growth of a sequence $a_1,a_2,\dots$ of $\ZN^m\times n$ is bounded by $f$ if $\deg a_i\leq f(i)$ for all $i$. Let $\L_{f,m}^n$ be the maximal length of an antichain sequence of $\ZN^m\times n$ with degree growth bounded by $f$. We then have
$$T(m,n,\ell) =  2^{\L_{f,m}^n+1}\ell, \quad \text{ where } f(i)=2^i\ell.$$ 

The number $T$ first appeared in~\cite[Theorems 4.3 and 4.10]{Pierce} and is related to the axiomatization of the class of differentially closed fields with several commuting derivations (in arbitrary characteristic). Theorem 4.10 of \cite{Pierce} is one of the main tools used to prove the upper bound (\ref{newbound}). However, \cite{Pierce} only dealt with the existence of such a number, and no algorithm to compute it was considered. It is worth mentioning at this point that, in \cite{effectivebounds}, the number $T$ together with \cite[Theorem 4.10]{Pierce}  have also been used to compute Bezout-type estimates for systems of algebraic partial differential equations. There, \cite[\S3]{effectivebounds}, an algorithm to compute $T$ was presented for $m=1,2$.

The goal of this paper is to build an effective algorithm to determine the value of $T$ (we also prove an upper bound in closed form in terms of the Ackermann function, see Example \ref{ex:app1}). In Section~\ref{sec:dicksonian}, we obtain explicit upper bounds for lengths of Dicksonian sequences whose degree growth is bounded by a given function $f$. The proofs of our bounds for Dicksonian sequences are based on the ideas of \cite[Lemma~8]{DiffNull}. However, the proof of the latter contains an error in the way it refers to \cite[Proposition 1.1]{Moreno1}. Here, we correct this error and improve the statements. In Section~\ref{sec:antichains},  upper bounds for the length of antichain sequences are obtained. Furthermore, we provide an explicit recursive algorithm which computes the exact value of the maximal length of antichain sequences; more precisely, of~$\L_{f,m}^n$. Note that our results provide explicit bounds for any number $m$ of derivations, while currently explicit bounds are only known for $m=1,2$. Due to the discussion above, having these explicit bounds is crucial for the effective differential Nullstellensatz (\ref{newbound}) (and for Bezout-type estimates of algebraic PDE's). Of course, it is still desirable to determine how sharp the bound in (\ref{newbound}) is, or how much it can be improved. These are interesting and difficult questions, which we leave for future research,

The type of bounds discussed in this paper have been studied in combinatorics using general versions of Dickson's lemma. Their existence, together with constructive recursive algorithms, appear in \cite{Fig,McA,Moreno,Pierce,Seid}. For instance, in \cite{Moreno}, it is shown that the maximal possible length of Dicksonian sequences (and antichains) is primitive recursive in the bounding function and recursive, but not primitive recursive (if the function increases at least linearly), in $m$. The motivation of our statements is the need to find explicit expressions of such bounds to make them more applicable to designing efficient algorithms, and thus to have a better understanding of the complexity of the differential effective Nullstellensatz (and, consequently, of differential elimination).

\section{Bounds for Dicksonian sequences}\label{sec:dicksonian}
This section contains explicit upper bounds for lengths of Dicksonian sequences with \emph{growth rate} bounded by a given function. We provide several versions of the bounds so that more cases are covered. We start by introducing some terminology.
Let
$$\tau_1 = \left(\tau_1^1,\ldots,\tau_m^1\right),\ldots
,\tau_k = \left(\tau_1^k,\ldots,\tau_m^k\right)
$$
be a  sequence of $m$-tuples of nonnegative integers and $
f: \ZN \to \ZN$
be an arbitrary  function.
\begin{definition}  We say that the {\it max growth} of this sequence is {\it bounded by the function $f$} if, for all $i$, $1\leq i\leq k$, 
\begin{equation*}\label{ineq1}
\max\big\{\tau_1^i,\ldots,\tau_m^i\big\}\leq f(i).
\end{equation*}
We also say that the {\it degree growth} of this sequence is {\it bounded by the function $f$} if, for all $i$, $1\leq i\leq k$, 
\begin{equation*}\label{ineq11}
\deg \tau_i := \tau_1^i+\ldots+\tau_m^i\leq f(i).
\end{equation*}
\end{definition} 
\begin{remark} The sequences with bounded max growth are used in the bounds for the effective differential Nullstellensatz found in \cite{DiffNull}, see (\ref{oldbound}); while sequences with bounded degree growth are used for the improved bounds found in \cite{DiffNull2}, see (\ref{newbound}).
\end{remark}
The Ackermann function, which is used in our bound estimates, is defined as follows (see \cite[\S2]{Moreno1}, for instance):
$$
A(m,n) =
\begin{cases}
   A(0,n)=n+1& n \geq 0,\\
   A(m+1,0)=A(m,1)& m\geq 0,\\
   A(m+1,n+1)=A(m,A(m+1,n))& m,n\geq 0.
  \end{cases}
$$

\begin{proposition}\cite[Proposition 1.1]{Moreno} \label{p1}
   For all $m,h,k\geq 1$, if $\tau_1,\ldots,\tau_k$ is a Dicksonian sequence of $m$-tuples, such 
   that $$\deg \tau_i=h+i-1,\quad 1\leq i\leq k,$$
   then $$k \leq A(m,h-1)-h,$$
   and there exists such a Dicksonian sequence for which this bound is reached.
\end{proposition}

We will use the following notation:
\begin{itemize}
\item Let $L_{f,m}$ denote the maximal length of a Dicksonian sequence of $m$-tuples with max growth
is bounded by $f$.
\item Let  $l_{f,m}$  denote the maximal length of a Dicksonian sequence of $m$-tuples with degree growth is bounded by $f$.
\item For an increasing function $f:\ZN\to\ZN$, let $\lceil f^{-1}(x)\rceil$ be the least number $k$
such that $f(k)\geq x$.
\end{itemize}

Under certain assumptions on the growth of the function $f$, the following lemmas yield upper bounds for $L_{f,m}$ and $l_{f,m}$ in closed form in terms of the Ackermann function. The idea of the proofs is that if the function $f$ does not grow ``too fast'', one can reduce the problem to the one treated in Proposition~\ref{p1}. This kind of statements has already been considered in \cite[Lemma~8]{DiffNull}; however, the proof of that lemma contains an error in the way it refers to \cite[Proposition 1.1]{Moreno}. Our lemmas below can be considered as a correction and/or improvement of that lemma. The general case (arbitrary function $f$) has been considered in \cite{Moreno1}; an algorithm to compute the value of $l_{f,m}$ is provided there. However, in general, a closed form of this bound is not available. Also, \cite[Theorem~10]{Seid} can be viewed as being more general than Lemma~\ref{lexbound}, but, again, the bounds  are not given explicitly there. Our results below are justified by the convenience of having explicit expressions of the bounds; moreover, such expressions will be used in Section~\ref{subsec:dicksonian}.

\begin{lemma}\label{lexbound} For every increasing function $f:\ZN \to \ZN$ and $d\in\ZN$ such that $$d\cdot f(i+1)\geq (m+d)f(i),\quad i>0,$$ 
%for some $d\in\ZN$, 
if $$(m+d)f(i+1)\leq A(d,d\cdot f(i)-1),\quad i>0,$$ then 
\begin{align}\label{ineq2}
L_{f,m} < \big\lceil f^{-1}\big(A(m+d,(d+m)f(1)-1)/d\big)\big\rceil
\end{align}
\end{lemma}
\begin{proof}
Let
\begin{align}\label{seq11}
\tau_1 = \left(\tau_1^1,\ldots,\tau_m^1\right),\ldots,\;
\tau_k = \big(\tau_1^k,\ldots,\tau_m^k\big)
\end{align}
be a Dicksonian sequence whose max growth is bounded by $f$. We now construct, from \eqref{seq11}, a new sequence satisfying the conditions of 
Proposition~\ref{p1}. We will use the fact that $\deg \tau_i > 0$ for $1\leq i < k$.
Append to the first tuple $d$ new coordinates, each equal to $f(1)$,
obtaining the following $(m+d)$-tuple:
$$\left(\tau_1^1,\ldots,\tau_m^1,f(1),\ldots,f(1)\right),$$
whose degree is $$\deg \tau_1+ d\cdot f(1)\leq (m+d)f(1).$$
Let $\alpha_1:=\deg\tau_2+d\cdot f(2)-\deg\tau_1-d\cdot f(1)-1$. Note that
$$0\leq \alpha\leq (m+d)f(2)-d\cdot f(1)-1.$$ 
Now, add $\alpha$ many new $(m+d)$-tuples as follows: The first $m$ coordinates of these tuples are
$(\tau_1^1,\ldots,\tau_m^1)$, and the last $d$ coordinates form a Dicksonian
sequence of $d$-tuples, starting with $(f(1),\ldots,f(1))$, with the 
degree growing exactly by 1 at each step. 
From Proposition~\ref{p1} and the condition
$$(m+d)f(2)-d\cdot f(1)\leq A(d,d\cdot f(1)-1)-d\cdot f(1),$$ such a sequence exists. The last tuple will 
have degree equal to $$\deg\tau_2+d\cdot f(2)-1.$$
Next, add the tuple
$$\left(\tau_1^2,\ldots,\tau_m^2,f(2),\ldots,f(2)\right).$$
Continue by adding new $(m+d)$-tuples, whose
first $m$ coordinates are $(\tau_1^i,\ldots,\tau_m^i)$ and last
$d$ coordinates form a Dicksonian sequence with the degree growing by 1 at
each step.
When the tuple
$$\big(\tau_1^{k-1},\ldots,\tau_m^{k-1},f(k-1),\ldots,f(k-1)\big)$$
is reached, consider two cases: (1) if $\deg \tau_k+d\cdot f(k)=\deg \tau_{k-1}+d\cdot f(k-1)$ then stop this construction; (2) otherwise, repeat the construction one more time and stop at $\big(\tau_1^{k},\ldots,\tau_m^{k},f(k),\ldots,f(k)\big)$. 
In both cases, we obtain a sequence of $(m+d)$-tuples in which the 
degree grows by 1 at each step. We will show
that this sequence is Dicksonian.
Suppose that it is not. 
Let $\tau_j$, $\tau_l$, $j<l$, be two $(m+d)$-tuples from this sequence 
for which there exists an $(m+d)$-tuple $\tau$ of nonnegative integers such that $$\tau_l=\tau_j+\tau.$$ 
For an $(m+d)$-tuple $\gamma$, let
 $\gamma'$ and $\gamma''$ denote the first $m$ 
coordinates and the last $d$ coordinates of it, 
respectively. Then we have $$\tau_l'=\tau_j'+\tau'\quad \text{and}\quad 
\tau_l''=\tau_j''+\tau''.$$
If $\tau_j$ and $\tau_l$ have been added after the same tuple 
of the form $$p_i=\left(\tau_1^i,\ldots,\tau_m^i,f(i),\ldots,f(i)\right),$$
or if $\tau_j$ coincides with such a tuple $p_i$ and $\tau_l$ has been 
added after $p_i$, 
the equality $$\tau_l''=\tau_j''+\tau''$$ contradicts the fact that the 
last $d$ coordinates of the tuples between $p_i$ and $p_{i+1}$,
including $p_i$ and excluding $p_{i+1}$, form a Dicksonian sequence.
If $\tau_j$ and $\tau_l$ have been added after {\it different} tuples
$p_i$ and $p_{i'}$, the equality
$$\tau_l'=\tau_j'+\tau'$$  contradicts the fact that sequence
\eqref{seq11} is Dicksonian.
Therefore, our assumption was false and the constructed sequence
is Dicksonian. 

%and the number of elements in the constructed sequence does not 
%exceed $$A(n+d,\deg\tau_1+d\cdot f(1)-1)-\deg\tau_1-d\cdot f(1).$$ 
%On the other hand, the number of elements in the constructed
%sequence is:
%\begin{align*}
%&\deg\tau_2+d\cdot f(2)-\deg\tau_1-d\cdot f(1)+\deg\tau_3+d\cdot f(3)-\deg\tau_2-d\cdot f(2)+\ldots + 
%\\&\quad\quad\deg\tau_k+d\cdot f(k)-\deg\tau_{k-1}-d\cdot f(k-1)+1=\\
%&\deg\tau_k+d\cdot f(k)-d\cdot f(1)-\deg\tau_1+1.
%\end{align*}
%Therefore,
%$$\deg\tau_k+d\cdot f(k)-d\cdot f(1)-\deg\tau_1+1\Le A(n+d,\deg\tau_1+d\cdot f(1)-1)-\deg\tau_1-d\cdot f(1).$$

By Proposition~\ref{p1}, the degree
of its last element does not exceed $$A(m+d,\deg\bar a_1+d\cdot f(1)-1)-1< A(m+d, (m+d)f(1)-1),$$ and, moreover, this degree equals $\deg\tau_k+d\cdot f(k)$. Hence, 
$$d\cdot f(k)< A(m+d,(d+m)f(1)-1),$$
and 
$$k<\big\lceil f^{-1}\big(A(m+d,(d+m)f(1)-1)/d\big)\big\rceil.\qedhere$$
\end{proof}

\begin{lemma}\label{lexbound2} For every increasing function $f:\ZN \to \ZN$ and $d\in\ZN$ such that $$f(i+1)\geq (m+1)f(i),\quad i>0,$$ 
%for some $d\in\ZN$, 
if $$(m+1)f(i+1)\leq A(d,f(i)-1),\quad i>0,$$ then 
\begin{align}\label{ineq22}
L_{f,m} < \big\lceil f^{-1}\big(A(m+d,(m+1)f(1)-1)\big)\big\rceil
\end{align} 
\end{lemma}
\begin{proof}
The proof follows as in Lemma~\ref{lexbound} where the appended $d$-tuples begin with the form 
$$(\tau_1^i,\dots,\tau_m^i,f(i),0,\ldots,0). \qedhere$$
\end{proof}

\begin{lemma}[cf. Proposition~\ref{antiprop}]\label{lexbound3} For every increasing function $f:\ZN \to \ZN$ and $a,d\in\ZN$ such that $$a\cdot f(i+1)\geq (a+1)f(i),\quad i>0,$$ 
%for some $a, d\in\ZN$, 
if $$(a+1)f(i+1)\leq A(d,a\cdot f(i)-1),\quad i>0,$$ then 
\begin{align}\label{ineq23}
l_{f,m} < \big\lceil f^{-1}\big(A(m+d,(a+1)f(1)-1)/a\big)\big\rceil
\end{align}
\end{lemma}
\begin{proof}
The proof follows as in Lemma~\ref{lexbound} where the appended $d$-tuples begin with the form 
$$(\tau_1^i,\dots,\tau_m^i,a\cdot f(i),0,\ldots,0). \qedhere$$
\end{proof}

\begin{example}\label{ex:bounds}
Let $d=2$, $b>1$, and $\ell>0$. Also, let $f(i)=b^i\ell$. Consider the following question: For which values of $b$ is there $a\in \ZN$ satisfying the hypothesis of Lemma~\ref{lexbound3}? We first have the inequality
$$ab^{i+1}\ell\geq (a+1) b^i\ell,$$
which is the same as $ab\geq a+1$. We also have
$$(a+1)b^{i+1}\leq A(2,ab^i-1)=2ab^i+1,\quad \text{ for all } i\geq 0.$$
This is equivalent to $(a+1)b\leq 2a$. From this inequality, we see that
$$b\leq \frac{2a}{a+1} < 2.$$
Moreover, from the above inequalities, we see that for any $1<b<2$, if $a\in\ZN$ is such that 
$$a\geq \frac{1}{b-1} \quad \text{ and }\quad a\geq\frac{b}{2-b},$$
then the hypothesis of Lemma~\ref{lexbound3} are satisfied. Thus, for such values of $a$ and $b$, we have
$$l_{f,m}< \big\lceil \log_{b}(A(m+2,(a+1)b\ell-1)/a\ell)\big\rceil.$$
In particular, if $b=3/2$ and $\ell=2/3$, we can choose $a=3$. In case $m=2$, we get
$$l_{f,2}<\big\lceil \log_{\frac{3}{2}}(A(4,3)/2)\big\rceil=\big\lceil\log_{\frac{3}{2}}((2^{2^{65536}}-3)/2)\big\rceil.$$
\end{example}

%\begin{example}\label{ex:bounds} \todo{ For any $b$, $4/3 < b < 2$, let $f(i+1) = b\cdot f(i)$. Then, if we set $d = 2$, we find $a> 0$ such that
%$$
%ab>(a+1)\quad\text{and}\quad (a+1)bf(i) \Le 2(af(i)-1)+3 = 2af(i)+1.
%$$
%Let $a = 2/(2-b)$. Then we have
%$$
%\left(\frac{2}{2-b}+1\right)bf(i)-\frac{4f(i)}{2-b}-1= \frac{4bf(i)-b^2f(i)-4f(i)}{2-b}-1\Le\frac{-f(i)(b-2)^2}{2-b}\Le 0.
%$$
%Moreover,
%$$
%\frac{2b}{2-b} - \frac{2}{2-b}-1=\frac{2b-2-2+b}{2-b} =\frac{3b-4}{2-b}> 0.$$
%Hence, by Lemma~\ref{lexbound3} (for simplicity, not using $f^{-1}$ and the division), we have
%$$
%l_{f,n} < A\left(n+2,\frac{(4-b)f(1)}{2-b}-1\right).
%$$
%If $n=2$, the above yields 
%$$
%l_{f,n}< \underbrace{2^{\iddots^2}}_{\frac{(4-b)f(1)}{2-b}+2},
%$$
%which is $A(4,4) = 2^{2^{2^{65536}}}$ if $f(1)=1$ and $b = 3/2$. I guess one has to take the ceiling function inside - can do later, after we agree that, in general, this is correct.}
%\end{example}

\begin{remark}
In the previous example we saw that $1<b<2$. Thus, Lemma~\ref{lexbound3} can only deal with the case $f(i)=2^i\ell$ when $d\geq 3$ (see Example~\ref{ex:app1}). As we saw in the introduction, the increasing function $f(i)=2^i\ell$ plays an important role in the applications of our bounds to the effective differential Nullstellensatz, and so better bounds are desirable. We deal with these issues in the next section.
\end{remark}

\section{Bounds for antichains}\label{sec:antichains}

In this section, we establish explicit bounds for lengths of antichain sequences of tuples of nonnegative integers, which can be used for computations of the bound (\ref{newbound}) obtained in \cite[Theorem 3.4]{DiffNull2} (see Example~\ref{ex:app}). Clearly, every such sequence is a Dicksonian sequence, and so the bounds obtained in the previous section can be applied; however, the goal of this section is to show that in general the maximal length of an antichain sequence is much smaller, and so better bounds can be obtained for the differential Nullstellensatz computations.

Let us recall some of the notation used in the introduction. Let $m$ and $n$ be positive integers\footnote{In this section, we consider $m$-tuples and $n$ copies of $\ZN^m$.}. Consider the order $\leq$ on $\ZN^m\times n$ defined as $(\tau,i)\leq (\eta,j)$ iff $i=j$ and $\tau$ is less than or equal to $\eta$ in the product order of $\ZN^m$. Recall that an antichain sequence of $\ZN\times n$ is a sequence $a_1,\dots,a_k$ such that for all $i\neq j$ we have that $a_i\nleq a_j$. If $a\in \NN^m\times n$, we let the degree of $a$ be $\deg a=\deg\tau$ where $a=(\tau,i)$.

Given $f:\ZN\to\ZN$, we let $\L_{f,m}^n$ be the maximal length of an antichain sequence of $\ZN^m\times n$ with degree growth bounded by $f$. In the following sections we find an upper bound of $\L_{f,m}^n$ in terms of the Ackermann function and, more importantly, we find a recursive algorithm which yields its exact value. Recall that, for a nonnegative integer $\ell$, the number $T=T(m,n,\ell)$ that appears in the bound (\ref{newbound}) is given by 
$$T=  2^{\L_{f,m}^n+1}\ell, \quad \text{ where } f(i)=2^i\ell.$$

\subsection{Using Dicksonian sequences}\label{subsec:dicksonian}
Using the results of Section \ref{sec:dicksonian}, we provide an upper bound for $\L_{f,m}^n$ (for a certain family of functions) in terms of the Ackermann function.

\begin{proposition}\label{antiprop}
Let $f:\NN\to \NN$ be an increasing function such that $f(i+1)\geq 2f(i)$. If $d\in\NN$ is such that 
$$2f(i+1)\leq A(d,f(i)-1),$$
then 
\begin{equation}\label{anb1}
\L_{f,m}^{1}< \big\lceil f^{-1}(A(m+d,2f(1)-1))\big\rceil.
\end{equation}
Moreover, if $d\geq 3$, then, for $n>1$,
\begin{equation}\label{anb2}
\L_{f,m}^{n}< \big\lceil f^{-1}(A(m+d+2,2f(1) n-1)/n)\big\rceil.
\end{equation}
\end{proposition}
\begin{proof}
Inequality \eqref{anb1} follows immediately from Lemma \ref{lexbound3} taking $a=1$ and noting that $\L_{f,m}^1\leq l_{f,m}$. For \eqref{anb2}, assume $n>1$. We can embed $\NN^m\times n$ into $\NN^{m+2}$ by 
\begin{equation}\label{emb}
(\xi,i)\mapsto (\xi,\, i,\, n-1-i). 
\end{equation}
Note that, given an increasing function $f:\NN\to\NN$ and a sequence $(a_1,\ldots,a_k)$ of elements in $\ZN^m\times n$, if $\deg a_i \leq f(i)$, then
$$\deg a_i+n-1\leq f(i) + n-1\leq F(i):= f(i)\cdot n.$$
Since $d\geq 3$,  we have
$$
2F(i+1)=2nf(i+1)\le nA(d,f(i)-1)\le A(d,F(i)-1).
$$
Thus, we get that $$F(\L_{F,m+2}^1)<A(m+d+2,2F(1)-1),$$ and so 
$$\L_{F,m+2}^1< \big\lceil f^{-1}(A(m+d+2,2f(1) n-1)/n)\big\rceil.$$
Furthermore, any antichain sequence of $\ZN^m\times n$ whose degree growth is bounded by $f$ yields, by means of the embedding \eqref{emb}, an antichain sequence of $\ZN^{m+2}$ whose degree growth is bounded by $F$. Therefore, $\L_{f,m}^n\leq \L_{F,m+2}^1$, and the result follows.
\end{proof}

\begin{example}\label{ex:app1}
In our case of interest for the applications to the effective differential Nullstellensatz, we have $f(i)=2^i\ell$ and, in this case, we write $\L_{\ell,m}^n$ instead of $\L_{f,m}^n$. In this case, we see that $f(i+1)\geq 2f(i)$ and, if we let $d=3$, we see that 
$$2f(i+1)\leq 2^{f(i)+2}-3=A(3,f(i)-1).$$
Thus, we can apply Proposition \ref{antiprop} to get that
$$2^{\L_{\ell,m}^{1}}\ell < A(m+3, 4\ell-1),$$ 
and, if $n>1$, then
$$2^{\L_{\ell,m}^{n}}\ell < A(m+5, 4n\ell-1)/n.$$
%$$2^{\L_{\ell,m}^{n}}\ell < \todo{A(m+5, 4\ell+2n-3)/n}.$$
We conclude that the value of $T$ in (\ref{newbound}) satisfies 
$$T < 2\, A(m+3,4\ell-1) \text{ when }n=1$$
and 
$$T<\frac{2}{n}A(m+5,4n\ell-1) \text{ when }n>1.$$
%$$\todo{T<\frac{2}{n}A(m+5,4\ell+2n-3) \text{ when }n>1}.$$
\end{example}

\subsection{Sequence giving the exact bound}
We now provide a recursive algorithm that yields the exact value of $\L_{f,m}^n$. The techniques we use are motivated by the arguments for Dicksonian sequences from \cite[\S4]{Moreno1}.
% and, in fact, some of our results could be (nontrivially) deduced from Moreno's results. However, the arguments for antichain sequences are much simpler than the corresponding ones for Dicksonian sequences, and thus we provide all necessary details, keeping the paper mostly self-contained. 

Clearly, $\L_{f,1}^n=n$. For $m=2$, we have (see   \cite[Lemma~3.8]{effectivebounds})
\begin{equation}\label{eq22}
\L_{f,2}^n=b_n,\quad \text{where } b_0:=0 \text{ and } b_{i+1}:=f(b_i+1)+b_i+1,\ i\ge 0.
\end{equation}
However, for $m>2$, the arguments in \cite[Section~3]{effectivebounds} do not yield the value of $\L_{f,m}^n$. 

For the rest of this section we assume that
\begin{enumerate}
\item [(\dag)] the bound function $f$ is increasing.
\end{enumerate}

\begin{remark}
If $f$ grows at least linearly, Proposition \ref{p1} yields a Dicksonian sequence of $\ZN^m$ of length $A(m,0)-1$ such that the degree grows (by one) at each step. Hence, this Dicksonian sequence is in fact an antichain sequence with degree growth bounded by $f$, and so
$$A(m,0)-1\leq \L_{f,m}^{n}.$$
\end{remark}

Let us now recall the notions of compressed sets and binomial representations (see \cite[\S1]{Moreno1} for more details). Consider the degree-lexicographic order in $\ZN^m$, $\tau\prec \eta$ iff $\deg \tau < \deg \eta$ or $\deg \tau= \deg \eta$ and $\tau$ is less than $\eta$ in the lexicographic order of $\ZN^m$. A subset $M$ of $\ZN^m$ is said to be \emph{compressed} if whenever $\tau, \eta \in \ZN^m$ and $\deg \tau=\deg \eta$ we have
$$\left( \tau \in M \text{ and } \tau\prec \eta\right) \,\implies \,  \eta\in M.$$
For $\gamma>0$, $k<\gamma$, and $a_0\geq a_1\geq \cdots\geq a_k\geq 0$ we define
$$\langle a_0,\dots,a_k\rangle _\gamma=\binom{a_0+\gamma}{\gamma}+\cdots +\binom{a_k+\gamma-k}{\gamma-k}.$$
For each $\gamma>0$, the map $(a_0,\dots,a_k)\mapsto \langle a_0,\dots,a_k\rangle_\gamma$ is an order-preserving bijection between decreasing sequences of $\ZN$ of length at most $\gamma$ (with the lexicographic order) and the positive integers (with the usual order). Thus, for every positive integer the inverse of this map yields a unique decreasing sequence which we call its \emph{$\gamma$-binomial representation}. For every positive integer $a$, we define
$$a^{\langle \gamma\rangle}=\langle a_0,\dots,a_k\rangle_{\gamma+1},$$
where $(a_0,\dots,a_k)$ is the $\gamma$-binomial representation of $a$. We set $0^{\langle\gamma\rangle}=0$. We have the following, if $a_1,a_2,b_1,b_2$ are nonnegative integers such that $b_1\geq a_1,a_2$,
\begin{equation}\label{bineq1}
a_1+a_2\leq b_1+b_2 \; \implies \; a_1^{\langle\gamma\rangle} + a_2^{\langle\gamma\rangle}\leq b_1^{\langle\gamma\rangle} + b_2^{\langle\gamma\rangle}.
\end{equation}

We now consider the analogue of the Hilbert-Samuel function for $\ZN^m$. Given a sequence $\tau=(\tau_1\dots,\tau_k)$ of $\ZN^m$, for each $i=0,1,\dots,k$ we let $H_{\tau}^i:\ZN\to\ZN$ be 
$$H_{\tau}^i(d)=\big|\{\xi\in \ZN^m : \deg \xi=d \text{ and } \xi\ngeq \tau_1,\dots, \tau_i\}\big|.$$
Recall that $\xi\ngeq \tau_j$ means that $\xi-\tau_j$ has at least one negative entry. Now, Macaulay's theorem on the Hilbert-Samuel function (cf. \cite[\S1]{Moreno1}) states that 
\begin{equation}\label{Mac1}
H_\tau^{i}(d+1)\leq H_\tau^{i}(d)^{\langle d\rangle}, \, \text{ for all } i=0,1,\dots,k \text{ and } d\geq 1.
\end{equation}
Moreover, if for some $i\leq k$ the sequence $(\tau_1,\dots,\tau_i)$ is compressed and $\deg \tau_j\leq d$ for all $j\leq i$, then 
\begin{equation}\label{Mac2}
H_\tau^{i}(d+1)=H_\tau^{i}(d)^{\langle d\rangle}.
\end{equation}

Let us now construct an algorithm that yields the values of $\L_{f,m}^n$. We first consider the case $n=1$. Our strategy is to build an appropriate antichain sequence and show that it has maximal length. The algorithm to compute its length will follow from the construction of such a sequence. We construct an antichain sequence as follows:
$$\mu_1=\max_{\prec}\{\xi\in \ZN^m : \deg \xi=f(1) \},$$
and, as long as it is possible, choose
\begin{equation}\label{rec0}
\mu_{i}=\max_{\prec}\{ \xi\in \ZN^m : \deg \xi=f(i) \text{ and } \xi \ngeq \mu_1,\dots,\mu_{i-1} \}.
\end{equation}
Since $f$ is increasing, $\bar \mu:=(\mu_1,\dots,\mu_L)$ is indeed an antichain sequence, and $f$ bounds its degree growth (in fact, $\deg \mu_i=f(i)$). We will show that $L=\L_{f,m}^1$. It is worth mentioning at this point that in \cite[\S4]{Moreno1} the value of $L$ is denoted by $\Omega(m,f)$ and is called the \emph{frontier} of $f$ in $\ZN^m$. 

Let us give a more explicit construction of $\bar \mu$. By the definition of $\prec$, the first element of $\bar \mu$ is 
$$\mu_1=(f(1),0,\dots,0),$$
if $f(1)>0$ (which we might as well assume), the second element is
$$\mu_2=(f(1)-1, f(2)-f(1)+1,0,\dots,0),$$
the third element is 
$$\mu_3=(f(1)-1,f(2)-f(1),f(3)-f(2)+1,0,\dots,0),$$
and so on. The penultimate element is 
$$\mu_{L-1}=(0,\dots,0,1,f(L-1)-1),$$
and the last element is
$$\mu_L=(0,\dots,0,f(L)).$$
More generally, a recursive construction of the sequence $\bar \mu$ is given as follows:
\begin{enumerate}
\item [(i)] if $\mu_{i-1}=(u_1,\dots, u_r,0,\dots,0,u_m)$ with $r<m-1$ and $u_r>0$, then
\begin{equation}\label{rec1}
\mu_i=(u_1,\dots,u_r-1,f(i)-f(i-1)+u_m+1,0,\dots,0)
\end{equation}
\item [(ii)] if $\mu_{i-1}=(u_1,\dots,u_{m-1},u_m)$ with $u_{m-1}>0$, then
\begin{equation}\label{rec2}
\mu_i=(u_1,\dots,u_{m-1}-1,f(i)-f(i-1)+u_m+1).
\end{equation}
\end{enumerate}

 \begin{remark}\label{onH}
From the recursive construction of $\bar \mu$, one sees that, for each $i\leq L$, the sequence $(\mu_1,\dots,\mu_i)$ is compressed and that $H_{\bar \mu}^L(\deg \mu_L)=0$.
\end{remark}

We now aim to show that $L=\L_{f,m}^1$. Let $\tau=(\tau_1,\dots,\tau_\nu)$ be an antichain sequence of $\ZN^m$ with degree growth bounded by $f$. Suppose $\nu\geq L$, we must show that then $\nu\leq L$. The first step is to replace $\tau$ with a more adequate antichain sequence of the same length.

\begin{lemma}\label{reord}
There is an antichain sequence $\eta=(\eta_1,\dots,\eta_\nu)$ with degree growth bounded by $f$, such that $\deg \eta_i\leq \deg \eta_{i+1}$ for $i=1,\dots,\nu$. In fact, one such a sequence can be obtained by reordering $\tau$.
\end{lemma}

\begin{remark}
The above lemma does not hold for Dicksonian sequences. This is an important difference between Dicksonian and antichain sequences.
\end{remark}

\begin{proof}
Let $\eta_1=\min_{\prec}\{\tau_1,\dots,\tau_\nu\}$ and $\eta_{i}=\min_{\prec}\{(\tau_1,\dots,\tau_\nu)\setminus (\eta_1,\dots,\eta_{i-1})\}$ for $i=2,\dots,\nu$. Clearly, $\eta=(\eta_1,\dots,\eta_\nu)$ is an antichain sequence (as $\tau$ is). Also, by construction, $$\deg \eta_i\leq \deg \eta_{i+1}.$$ Thus, all that is left to show is that $f$ bounds the degree of $\eta$. To see this, note that, by the definition of $\eta_i$, there must be $1\leq j\leq i$ such that $\deg \eta_i\leq \deg \tau_j$. But since $f$ is assumed to be increasing and it bounds the degree growth of $\tau$, we get $$\deg \eta_i\leq f(j)\leq f(i),$$ as desired.
\end{proof}

Now, let 
$$g(i)=\left\{\begin{array}{lc} \deg \eta_i\, , & \text{for } i\leq \nu \\ \deg \eta_\nu\, , & \text{for } i>\nu \end{array}\right.$$
Clearly, $g$ is an increasing function such that $g(i)\leq f(i)$. Let $\zeta=(\zeta_1,\dots,\zeta_M)$ be the antichain sequence with degree growth bounded by $g$ constructed as in \eqref{rec0}. 

\begin{lemma}
With $\bar \mu=(\mu_1,\cdots,\mu_L)$ and $\zeta=(\zeta_1,\dots,\zeta_M)$ as above, we have~$M\leq~L$.
\end{lemma}
\begin{proof}
By the construction of $\bar \mu$ (see \eqref{rec1} and \eqref{rec2}), we have that if $(u_1,\dots,u_m)\in \bar\mu$ and $0\leq v_1\leq u_1,\dots,0\leq v_{m-1}\leq u_{m-1}$, then 
\begin{equation}\label{eq:sharp}
 \text{ there exists }v_m\geq u_m \text{ such that } (v_1,\dots,v_m)\in\bar\mu.
 \end{equation}
We now prove

\noindent \underline{Claim.} For every $(u_1,\dots,u_{m-1},u_m)\in\zeta$, there is $(u_1,\dots,u_{m-1},v_m)\in \bar \mu$ with $v_m\geq u_m$. 

\noindent {\it Proof of Claim.} We proceed by induction on $i=1,\dots,M$. The first element of $\zeta$ is $(g(1),0\dots,0)$, but the first element of $\bar \mu$ is $(f(1),0\dots,0)$ and $g(1)\leq f(1)$, thus, by~\eqref{eq:sharp}, we can find the desired tuple in $\bar \mu$. Now suppose $$\zeta_{i-1}=(u_1,\dots,u_r,0,\dots,0,u_m),\quad r<m-1,\ u_r>0.$$ By induction, there is $v_m\geq u_m$ such that $(u_1,\dots,u_r,0,\dots,0,v_m)\in\bar \mu$. By \eqref{rec1}, $\zeta_i$ is of the form $$(u_1,\dots,u_r -1,\alpha,0,\dots,0), \quad \alpha:=g(i)-u_1-\cdots -u_r+1.$$ Also, by \eqref{rec1}, $$(u_1,\dots,u_r-1,\beta,0\dots,0)\in\bar\mu,\quad \beta:=f(i)-u_1-\cdots - u_r+1.$$ Since $g(i)\leq f(i)$, we have that $\alpha\leq \beta$, and so, by~\eqref{eq:sharp}, we can find the desired tuple in $\bar \mu$. Finally, suppose $$\zeta_{i-1}=(u_1,\dots,u_{m-1},u_m),\quad u_{m-1}>0.$$ By induction, there is $v_m\geq u_m$ such that $(u_1,\dots,u_{m-1},v_m)\in\bar \mu$. By \eqref{rec2}, $\zeta_i$ is of the form $$(u_1,\dots,u_{m-
1}-1,\alpha'),\quad\alpha':=g(i)-u_1-\cdots - u_{m-1}+1.$$ Also, by \eqref{rec2}, $$(u_1,\dots,u_{m-1}-1,\beta')\in\bar\mu,\quad \beta'=f(i)-u_1-\cdots - u_{m-1}+1.$$ Again, since $g(i)\leq f(i)$, we have $\alpha'\leq \beta'$, and so, once again, by $(\eqref{eq:sharp})$ we can find the desired tuple in $\bar\mu$. 
This proves the claim.

The claim implies that every element of $\zeta$ will be accounted for in $\bar \mu$, and so $M\leq L$.
\end{proof}

By the above lemma, $M\leq L\leq \nu$. Thus, it suffices to show that $\nu\leq M$. Note that $\deg \zeta_i=\deg \eta_i$ for all $i\leq M$. We now establish how $H_{\eta}^i$ is related to $H_{\zeta}^i$.

\begin{proposition}\label{ontheH}
For each $i=0,1,\dots,M$ and $d\geq 0$, we have $H_{\eta}^i(d)\leq H_{\zeta}^i(d)$.
\end{proposition}
\begin{proof}
We proceed by induction on $i$. For the base case $i=0$, we have
$$H_{\eta}^0(d)=\binom{m-1+d}{d}=H_\zeta^0 (d),$$
which is the number of $m$-tuples of degree $d$. We now proceed with the induction step $i+1$. We have that for $d<g(i+1)=\deg \eta_{i+1}=\deg \zeta_{i+1}$, 
$$H_\eta^{i+1}(d)=H_\eta^{i}(d)\leq H_\zeta^i(d)=H_\zeta^{i+1}(d).$$
For $d=g(i+1)$, we have
$$H_\eta^{i+1}(d)=H_\eta^i(d)-1\leq H_\zeta^i(d)-1=H_\zeta^{i+1}(d).$$
Now let $d\geq g(i+1)$. For this case we follow the strategy of the last part of the proof  of \cite[Proposition 4.3]{Moreno1}. By Macaulay's theorem on the Hilbert-Samuel function (see \eqref{Mac1}), 
\begin{equation}\label{use1}
H_\eta^{i+1}(d+1)\leq H_\eta^{i+1}(d)^{\langle d\rangle},
\end{equation}
As we pointed out in Remark~\ref{onH}, the sequence $(\zeta_1,\dots,\zeta_i)$  is compressed for all $i$, and so the theorem of Macaulay also yields (see \eqref{Mac2}) 
\begin{equation}\label{use2}
H_\zeta^{i+1}(d+1)=H_\zeta^{i+1}(d)^{\langle d\rangle}.
\end{equation}
It then follows, by induction on $d\geq g(i+1)$ and the fact that if $a\leq b$ then $a^{\langle d\rangle}\leq b^{\langle d\rangle}$ (which follows from \eqref{bineq1}), that 
\begin{equation}\label{use3}
H_\eta^{i+1}(d)^{\langle d\rangle}\leq H_\zeta^{i+1}(d)^{\langle d\rangle}.
\end{equation}
Thus, putting \eqref{use1}, \eqref{use2}, and \eqref{use3} together, we get
$$H_\eta^{i+1}(d+1)\leq H_\zeta^{i+1}(d+1),$$
and the result follows.
\end{proof}

By the above proposition, we have that 
$$H_{\eta}^M(\deg \eta_M)\leq H_{\zeta}^M(\deg \zeta_M).$$ 
As we pointed out in Remark~\ref{onH}, we have that $H_\zeta^M(\deg \zeta_M)=0$. Thus,  $H_{\eta}^M(d)=0$ for all $d\geq \deg \eta_M$. Now, if $\nu>M$, then $\deg \eta_{M+1}\geq \deg \eta_{M}$, and this would imply that 
$$H_{\eta}^M(\deg \eta_{M+1})> 0,$$
which contradicts the previous sentence, and so we get $\nu\leq M$. Thus, we have shown the following

\begin{theorem}\label{thebound}
If $f$ is an increasing function, then the antichain sequence $\bar \mu=(\mu_1,\dots,\mu_L)$ built above has the maximal length among antichain sequences of $\ZN^m$ with degree growth bounded by $f$, and so $L=\L_{f,m}^1$.
\end{theorem}

\subsection{Recursive construction for the length}
We can now give a recursive expression for $\L_{f,m}^1$ by giving such an expression for $L$. We remind the reader that we are working under the assumption that $f$ is increasing. From the recursive construction of $\bar \mu$, we observe that, to find its length $L$, we simply need to keep track of the number of steps in the above construction (cf. \eqref{rec1} and \eqref{rec2}), and note that we stop once we reach the tuple $(0,\dots,0,f(L))$. To do this, we let $i$ denote our counter. Consider $\Psi_{f,m}=\Psi:\NN\times\NN^{m}\to \NN$ given by
$$\Psi(i,(0,\dots,0,u_n))=i$$
with
\begin{align*}\Psi&(i-1,(u_1,\dots,u_r,0,\dots,0,u_m))\\&=\Psi(i,(u_1,\dots,u_r-1,f(i)-f(i-1)+u_m+1,0,\dots,0)),\quad r<m-1, \ u_r>0\end{align*}
 and 
\begin{align*}\Psi(i-1,(u_1,\dots,u_{m-1},u_m))=\Psi(i,(u_1,\dots,u_{m-1}-1,f(i)-f(i-1)+u_m+1)),\quad u_{m-1}>0. 
\end{align*}
Theorem \ref{thebound} yields that

\begin{corollary}\label{corn1}
If $f$ is an increasing function, then 
$$\L_{f,m}^1=\Psi_{m,f}(1,(f(1),0,\dots,0)).$$
\end{corollary}

\begin{remark}\label{fortwo}
A straightforward computation shows that, if $m=2$, then 
$$\L_{f,2}^1=\Psi(1,(f(1),0))=f(1)+1,$$ 
which is what one expects. 
\end{remark}

We now extend this recursive expression to $n>1$. As in the case $n=1$, we recursively build an antichain sequence of $\ZN^m\times n$ of maximal length. Again, we assume that $f$ is increasing. Let $\bar \mu^{(1)}$ be the antichain sequence with degree growth bounded by $f_0(x):=f(x)$ constructed in \eqref{rec0} inside of $\ZN^m\times\{0\}$. Let $L_1$ denote the length of $\bar \mu^{(1)}$; thus, $\bar\mu^{(1)}$ is of the form $$((\mu^{(1)}_1,0),\dots,(\mu^{(1)}_{L_1},0)).$$ Similarly, let $\bar \mu^{(2)}$ be the antichain sequence with degree growth bounded by $f_1(x):=f(x+L_1)$ constructed in \eqref{rec0} inside of $\ZN^m\times\{1\}$, and let $L_2$ be the length of $\bar\mu^{(2)}$. Then, $$\bar\mu^{(2)}=((\mu^{(2)}_1,1),\dots,(\mu^{(2)}_{L_2},1)).$$ Continuing in this fashion, we build $\bar \mu^{(i)}$ for $i=3,\dots n$ as the antichain sequence with degree bounded growth bounded by $$f_{i-1}(x)=f(x+L_1+\cdots+ L_{i-1})$$ constructed in \eqref{rec0} inside of $\ZN^m\times\{i-1\}$. It is easy 
to 
check that if $\bar \mu$ is the concatenation of $\bar \mu^{(1)},\dots,\bar \mu^{(n)}$, then $\bar \mu$ is an antichain sequence of $\ZN^m\times n$ with 
degree growth bounded by $f$. 

To prove that $\L_{f,m}^n=L_1+\dots+ L_n$, we will need the following technical lemma.

\begin{lemma}\label{techlem1}
Suppose $a_1,\dots,a_r$ and $b_1,\dots,b_s$ are sequences of nonnegative integers such that $b_1=\cdots=b_{s-1}\geq b_s$ and $b_1\geq a_i$ for all $i\leq r$. If $a_1+\cdots+a_r\leq b_1+\cdots+b_s$, then, for every $\gamma>0$, we have that
$$a_1^{\langle\gamma\rangle}+\cdots+a_r^{\langle\gamma\rangle}\leq b_1^{\langle\gamma\rangle}+\cdots+b_s^{\langle\gamma\rangle}.$$
\end{lemma}
\begin{proof}
We may assume that $a_1\geq \cdots\geq a_r$. We proceed by induction on $(r,s)$ using the lexicographic order. By \eqref{bineq1}, if $a\leq b$ then $a^{\langle\gamma\rangle}\leq b^{\langle\gamma\rangle}$. The case $r=1$ follows from this observation. It also follows from \eqref{bineq1} that $a^{\langle\gamma\rangle}+b^{\langle\gamma\rangle}\leq (a+b)^{\langle\gamma\rangle}$, and the case $s=1$ follows from this. Thus, we assume that $r,s>1$. We now consider two cases:

\noindent \underline{Case 1.} Suppose $b_s\geq a_r$. Then the sequences $a_1,\dots,a_{r-1}$ and $b_1,\dots,b_{s-1},b_s-a_r$ satisfy our hypothesis. By induction, 
 $$a_1^{\langle\gamma\rangle}+\cdots+a_{r-1}^{\langle\gamma\rangle} \leq b_1^{\langle\gamma\rangle}+\cdots+b_{s-1}^{\langle\gamma\rangle}+ (b_s-a_r)^{\langle \gamma\rangle}.$$
Using that $a_r^{\langle\gamma\rangle}+(b_s-a_r)^{\langle\gamma\rangle}\leq b_s^{\langle\gamma\rangle}$ (which follows from \eqref{bineq1}), we get the desired inequality for the original sequences.

\noindent \underline{Case 2.} Suppose $b_s< a_r$. When $s=2$, we must have that $a_1+\cdots +a_{r-1}\leq b_1$, and so
$$a_1^{\langle\gamma\rangle}+\cdots+a_r^{\langle\gamma\rangle}\leq (a_1+\cdots +a_{r-1})^{\langle\gamma\rangle}+a_r^{\langle\gamma\rangle}\leq b_1^{\langle\gamma\rangle}+b_2^{\langle\gamma\rangle},$$
where the latter inequality follows from \eqref{bineq1}. So we assume that $s>2$. If it happens that $a_1+\cdots+a_{r}\leq b_1+\cdots+b_{s-1}$, then we are done by induction. So we can assume that
\begin{equation}\label{cas}
a_1+\cdots+a_{r-1} > b_1+\cdots+b_{s-2}.
\end{equation}
We have that
$$b_{s-1} + b_s\geq a_r +\left(a_1+\cdots +a_{r-1}-b_1-\cdots -b_{s-2}\right).$$
It follows from \eqref{bineq1}, and using \eqref{cas}, that
$$b_{s-1}^{\langle\gamma\rangle} +b_s^{\langle\gamma\rangle} \geq a_r^{\langle\gamma\rangle} + \left(a_1+\cdots +a_{r-1}-b_1-\cdots -b_{s-2}\right)^{\langle\gamma\rangle}.$$
Thus, it suffices to see that 
$$b_1^{\langle\gamma\rangle}+\cdots + b_{s-2}^{\langle\gamma\rangle} + \left(a_1+\cdots +a_{r-1}-b_1-\cdots -b_{s-2}\right)^{\langle\gamma\rangle}\geq a_1^{\langle\gamma\rangle}+\cdots +a_{r-1}^{\langle\gamma\rangle},$$
but this follows by induction.
\end{proof}

\begin{proposition}\label{forn}
If $f$ is an increasing function, then the above antichain sequence $\bar \mu$ of $\ZN^m \times n$ has the maximal length among the antichain sequences of $\ZN^m\times n$ with degree growth bounded by $f$. In particular, $\L_{f,m}^n=L_1+\cdots+L_n$.
\end{proposition}
\begin{proof}
First enumerate $\bar \mu=(\mu_1,\dots,\mu_L)$, where $L=L_1+\cdots+L_n$. Let $\bar a=(a_1,\dots,a_\nu)$ be an antichain sequence of $\ZN^m\times n$ of degree growth bounded by $f$. Suppose $\nu\geq L$, we must show that $\nu\leq L$. We assume, by reordering $\bar a$ if necessary (as in Lemma~\ref{reord}), that 
$$\deg a_i\leq \deg a_{i+1},\quad i=1,\dots,\nu.$$ 
We also assume, by replacing $f$ and $\bar \mu$ if necessary (as in the discussion after Lemma \ref{reord} and prior to Proposition~\ref{ontheH}), that $\deg a_i=\deg \mu_i=f(i)$ for all $i=1,\dots,L$. 

Given a sequence $\bar b=(b_1,\dots,b_k)$ of $\ZN^m\times n$, for each $i=0,1,\dots,k$ we let $H_{\bar b}^i:\ZN\to\ZN$ be 
$$H_{\bar b}^i(d)=\big|\{c\in \ZN^m\times n : \deg c=d \text{ and } c\ngeq b_1,\dots, b_i\}\big|.$$
If, for each $1\leq j\leq n$, we let $H_{\bar b}^{i,j}$ be the Hilbert-Samuel function of the subsequence of $\bar b$ consisting of its elements in the $j$-copy of $\ZN^m$ (i.e., inside of $\ZN^m\times\{j-1\}$), then clearly
\begin{equation}\label{inth}
H_{\bar b}^i(d)=H_{\bar b}^{i,1}(d)+\cdots+H_{\bar b}^{i,n}(d).
\end{equation}
As in Remark \ref{onH}, by the construction of $\bar\mu$, we have that $H_{\bar \mu}^L(\deg \mu_L)=0$. Moreover, when $L_0+\cdots +L_{j-1} \leq i \leq L_0+\cdots +L_{j}$  (where $L_0=0$), by the construction of $\bar \mu$, we have that, for all $d\geq 0$,
\begin{align*}
0=H_{\bar\mu}^{i,1}(d)=\cdots=H_{\bar\mu}^{i,j-1}(d)\leq H_{\bar \mu}^{i,j}(d) \leq H_{\bar\mu}^{i,j+1}(d)=\cdots =H_{\bar\mu}^{i,n}(d)=\binom{m-1+d}{d},
\end{align*}
where the latter is the number of $m$-tuples of degree $d$. We now claim that, for all $i=1,\dots,L$ and $d\geq 0$, we have $H_{\bar a}^i(d)\leq H_{\bar \mu}^i(d)$. When $i=0$, we have
$$H_{\bar a}^0(d)=n\cdot\binom{m-1+d}{d}= H_{\bar \mu}^0(d).$$
Now, when $d<f(i+1)=\deg a_{i+1}=\deg \mu_{i+1}$, we have 
$$H_{\bar a}^{i+1}(d)=H_{\bar a}^{i}(d)\leq H_{\bar \mu}^i(d)=H_{\bar \mu}^{i+1}(d).$$
For $d=f(i+1)$, we have
$$H_{\bar a}^{i+1}(d)=H_{\bar a}^i(d)-1\leq H_{\bar \mu}^i(d)-1=H_{\bar \mu}^{i+1}(d).$$
Now let $d\geq f(i+1)$. By Macaulay's theorem on the Hilbert-Samuel function (see \eqref{Mac1}), we have that
\begin{equation}\label{usep1}
H_{\bar a}^{i+1,j}(d+1)\leq H_{\bar a}^{i+1,j}(d)^{\langle d\rangle},
\end{equation}
Since for each $j=1,\dots,n$ and $i=1,\dots, L_j$, the sequence $(\mu^{(j)}_1,\dots,\mu^{(j)}_i)$  is compressed, the theorem of Macaulay also yields (see \eqref{Mac2}) 
\begin{equation}\label{usep2}
H_{\bar \mu}^{i+1,j}(d+1)=H_{\bar \mu}^{i+1,j}(d)^{\langle d\rangle}.
\end{equation}
It then follows, by induction on $d\geq f(i+1)$ and using Lemma \ref{techlem1}, that 
\begin{equation}\label{usep3}
H_{\bar a}^{i+1,1}(d)^{\langle d\rangle}+\cdots+H_{\bar a}^{i+1,n}(d)^{\langle d\rangle}\leq H_{\bar \mu}^{i+1,1}(d)^{\langle d\rangle}+\cdots +H_{\bar \mu}^{i+1,n}(d)^{\langle d\rangle}.
\end{equation}
Thus, putting \eqref{inth}, \eqref{usep1}, \eqref{usep2}, and \eqref{usep3} together, we get
$$H_{\bar a}^{i+1}(d+1)\leq H_{\bar \mu}^{i+1}(d+1).$$
This proves our claim. The result now follows as in the discussion prior to Theorem~\ref{thebound}.
\end{proof}

By Proposition \ref{forn} and the recursive construction of $\bar \mu$, we obtain

\begin{corollary}\label{cor:bound}
If $f$ is an increasing function, then
$$\L_{f,m}^n=\psi_n$$
where $\psi_0=0$ and $\psi_{i+1}=\Psi_{f_{i},m}(1,(f_i(1),0,\dots, 0))+\psi_i$, for $i\geq0$ . We recall that $f_i(x)=f(x+\psi_i)$.
\end{corollary}

By Remark \ref{fortwo}, when $m=2$, we get precisely \eqref{eq22}, as expected.

\begin{example}\label{ex:app}
Again, for the applications to the effective differential Nullstellensatz we consider $f(i)=2^i\ell$ and, in this case, we write $\L_{\ell,m}^n$ instead of $\L_{f,m}^n$. 
% \todo{The problem is: the universal constant in the diffnull bound has not been determined. So, O(number) do not provide much information. Although, it would still make sense to put that constant in here and say that it will be determined later. It is not so easy to find that constant, because this would require digging into the various algorithms it is coming from. We might also improve the bound some time in the near future, and determining the constant would become less important. So, I propose to write this either with $C$ or to just calculate the $T$. Do you agree? Which of the two do you prefer?}
\begin{enumerate}
 \item When $m=2$ and $n=1$, we get $\L_{\ell,2}^1=2\ell+1$. So, in this case, the value of $T$ is 
$$T=2^{2\ell +2}\ell.$$
% If $\ell=1$, then $T=16$. In this case, the upper bound in the differential Nullstellensatz (Theorem~\ref{nulltheo}) becomes
% $$(\alpha_{15}D)^{2^{O\big( \alpha_{16}^3\big)}}=(136\cdot D)^{2^{O\big(3581577\big)}}.$$

\item More generally, when $m=2$, we get $\L_{\ell,2}^n=b_n$ where $b_0=0$ and $b_{i+1}=2^{b_i +1}\ell +b_i +1$. For instance,
$$\L_{\ell,2}^2=2^{2\ell+2}\ell +2\ell+2 \text{ and } \L_{\ell,2}^3= 2^{2^{2\ell+2}+2\ell+3}\ell+2^{2\ell +2}+2\ell+3.$$

\item For $m=3$, up until now no explicit values of $T$ were known. Let $n=1$ and $\ell=1$, then 
$$\L_{1,3}^1=\Psi_{f,3}(1,(2,0,0))=70,$$
and so, in this case, $T=2^{71}$. Note that the antichain sequence of maximal length built in \eqref{rec0} takes the form
\vspace{.1in}

\noindent $(2,0,0)$

\vspace{.05in}
\noindent $(1,3,0), (1,2,5), (1,1, 14), (1,0,31)$ 

\vspace{.05in}
\noindent $(0,64,0), (0,63,2^7 -63), (0,62, 2^8-62),\ldots , (0,1,2^{69}-1), (0,0,2^{70})$
\vspace{.1in}
 
% The bound for the differential Nullstellensatz becomes
% $$(\alpha_{2^{71}-1}D)^{2^{O\big( \alpha_{2^{71}}^3\big)}}.$$

\item For $m=3$, $n=1$ and $\ell=2$, we get $\L_{2,3}^1= 2^{2^{520}+520}+2^{520}+519$. So the value of $T$ is
$$T=2^{2^{2^{520}+520}+2^{520}+521}.$$
In this case, the antichain of maximal length takes the form
\vspace{.1in}

\noindent $(4,0,0)$ 

\vspace{.05in}
\noindent $(3,5,0), (3,4,9), \dots, (3,0,2^8-3)$ 

\vspace{.05in}
\noindent $(2,2^{9}-2,0), (2,2^9-3,2^{10}-2^9+1),\dots, (2,0,2^{2^9 +7}-2)$ 

\vspace{.05in}
\noindent $(1,2^{2^9 +8}-1,0),(1,2^{2^9+8}-2,2^{2^9+9}-2^{2^9+8}+1),\dots,(1,0,2^{2^{2^9+8}+2^9+7}-1)$ 

\vspace{.05in}
\noindent $(0,2^{2^{2^9+8}+2^9+8},0),(0,2^{2^{2^9+8}+2^9+8}-1,2^{2^{2^9+8}+2^9+9}-2^{2^{2^9+8}+2^9+8}+1),\dots,(0,0,2^{2^{2^{2^9+8}2^9+8}+2^{2^9+8}+2^9+8})$ 

\end{enumerate}
\end{example}

\setlength{\bibsep}{6pt}
\bibliographystyle{model1b-num-names}
\bibliography{bibdata}
\end{document}